\documentclass{amsart}

\usepackage[ruled,vlined, algo2e]{algorithm2e}
\usepackage{amsthm}
\usepackage{amsmath,latexsym}
\usepackage{amssymb}
\usepackage{setspace}
\usepackage[english, activeacute]{babel}
\usepackage{float}
\usepackage{hyperref}

\usepackage[all]{xy}

\newcommand{\PF}{\mathrm{PF}}
\newcommand{\SG}{\mathrm{SG}}
\newcommand{\g}{\mathrm{g}}
\newcommand{\F}{\mathrm{F}}
\newcommand{\G}{\mathrm{G}}
\newcommand{\m}{\mathrm{m}}
\newcommand{\e}{\mathrm{e}}
\newcommand{\Ap}{\mathrm{Ap}}
\newcommand{\Z}{\mathbb{Z}}
\newcommand{\N}{\mathbb{N}}

\newcommand{\C}{\mathrm{C}}

\newcommand{\I}{\mathcal{I}}
\newcommand{\Kunz}{\mathcal{K}}
\newcommand{\dsum}{\displaystyle\sum}

\newtheorem{theo}{Theorem}
\newtheorem{ex}[theo]{Example}
\newtheorem{lem}[theo]{Lemma}
\newtheorem{cor}[theo]{Corollary}
\newtheorem{prop}[theo]{Proposition}

\begin{document}

\title{The tree of irreducible numerical semigroups with fixed Frobenius number}

\author[V. Blanco and J.C. Rosales]{V. Blanco and J.C. Rosales\medskip\\
D\lowercase{epartamento de } \'A\lowercase{lgebra}, U\lowercase{niversidad de } G\lowercase{ranada}\medskip\\
\texttt{\lowercase{vblanco@ugr.es}} \; \texttt{\lowercase{jrosales@ugr.es}}}

\address{Departamento de \'Algebra, Universidad de Granada}

\begin{abstract}
In this paper we present a procedure to build the set of irreducible numerical semigroups with a fixed Frobenius number. The construction gives us a rooted tree structure for this set. Furthermore, by using the notion of Kunz-coordinates vector we translate the problem of finding such a tree into the problem of manipulating $0-1$ vectors with as many component as the Frobenius number. %Finally we apply the construction to the problem of finding minimal decompositions of a numerical semigroups into irreducible ones.
\end{abstract}

\maketitle

\section{Introduction}
A numerical semigroup is a subset $S$ of $\N$ (here $\N$ denotes the set of nonnegative integers) closed under
addition, containing zero and such that $\N\backslash S$ is finite. The largest integer not belonging to $S$ is called the \textit{Frobenius number} of $S$ and we denote it by $\F(S)$.

A numerical semigroup is \textit{irreducible} if it cannot be expressed as an intersection of two numerical semigroups containing it properly. This notion was introduced in \cite{pacific} where
it is also proven that a numerical semigroup is irreducible if and only if $S$ is maximal (with respect to the inclusion ordering) in the set of numerical semigroups with Frobenius number $\F(S)$.  From \cite{barucci} and \cite{froberg} it is also deduced in \cite{pacific} that the family of irreducible numerical semigroups is the union of two families of numerical semigroups with special importance in this theory: symmetric and pseudo-symmetric numerical semigroups.

For a given positive integer $F$, we denote by $\I(F)$ the set of irreducible numerical semigroups with Frobenius number $F$. The main goal of this paper is to show how to build, iteratively, all the elements in $\I(F)$. Furthermore, such a construction will be represented as a rooted tree.

For the sake of simplicity of the procedure we use the notion of \emph{Kunz-coordinates vector}, previously introduced in \cite{bp2011}, to encode the numerical semigroups as $0-1$ vectors with as many coordinates as its Frobenius number. The equivalence between the numerical coordinate and its Kunz-coordinates vector allows us to translate the construction of the tree of irreducible numerical semigroups with fixed Frobenius number into an easy procedure by manipulating $0-1$ vectors.

The paper is organized as follows. We present the construction of the tree of numerical semigroups with fixed Frobenius number in Section \ref{sec:1}. In Section \ref{sec:2} we give the notion of Kunz-coordinates vector of a numerical semigroups and translate the problem of building the tree into the problem of finding adequate $0-1$ vectors. %Finally, in Section \ref{sec:3} we apply the above results to compute a minimal decomposition of a numerical semigroups into irreducible numerical semigroups. 
 It leads us to an efficient procedure to compute the set of irreducible numerical semigroups with Frobenius number $F$ by swapping elements in a vector with components in $\{0,1\}$.

\section{The tree of irreducible numerical semigroups with Frobenius number $F$}
\label{sec:1}
If $A$ is a nonempty subset in $\N$, we denote by $\langle A \rangle$ the submonoid of $(\N, +)$ generated by $A$, that is $\langle A \rangle = \{\lambda_1a_1 + \cdots + \lambda_na_n: n\in \N \backslash \{0\}, a_1, \ldots, a_n \in A, \text{ and } \lambda_1, \ldots, \lambda_n \in \N\}$. It is well-known that $\langle A \rangle$ is a numerical semigroup if and only if $\gcd(A)=1$. If $S$ is a numerical semigroup and $S=\langle A \rangle$, then we say that $A$ is a system of generators of $S$. In \cite{springer} it is shown that every numerical semigroup admits an unique minimal system of generators and that such a system is finite.

The following result has an immediate proof.

\begin{lem}
\label{lem:1}
 Let $S$ be a numerical semigroup and $x$ a minimal generator of $S$. Then $S\backslash \{x\}$ is a numerical semigroups.
\end{lem}

Let $S$ be a numerical semigroup. Following the notation in \cite{jpaa}, we say that an integer $x\in \Z\backslash S$ is a pseudo-Frobenius number of $S$ if $x+s \in S$ for all $s\in S\backslash \{0\}$. We denote by $\PF(S)$ the set of pseudo-Frobenius numbers of $S$.

From lemmas 20 and 24 in \cite{simetricos}, lemma 27 in \cite{pseudo-simetricos} and the fact that a numerical semigroup is irreducible if and only if it is symmetric or pseudo-symmetric, we get the following result.

\begin{lem}
\label{lem:2}
 Let $S$ be an irreducible numerical semigroup and $x$ a minimal generator of $S$ such that $x < \F(S)$. Then, $\F(S)-x \in  \PF(S\backslash \{x\})$ if and only if $2x-\F(S) \not\in S$.
\end{lem}

There are many characterizations of symmetric and pseudo-symmetric numerical semigroups. Probably, the most used is the following that appears in \cite{springer}.

\begin{lem}
\label{lem:3}
 Let $S$ be a numerical semigroup. Then:
\begin{enumerate}
 \item $S$ is symmetric if $\F(S)$ is odd and if $x \in \Z\backslash S$, then $\F(S)-x \in S$.
 \item $S$ is pseudo-symmetric if $\F(S)$ is even and if $x \in \Z\backslash S$, then either $x = \frac{\F(S)}{2}$ or $\F(S)-x \in S$.
\end{enumerate}

\end{lem}

Let $S$ be a numerical semigroup. The cardinal of its set of gaps, $\G(S)=\N \backslash S$, is called the genus of $S$ and it is usually denoted by $\g(S)$.

\begin{lem}
\label{lem:4}
 Let $S$ be a numerical semigroup. Then:
\begin{enumerate}
 \item $S$ is symmetric if and only if $\g(S)=\frac{\F(S)+1}{2}$.
\item  $S$ is pseudo-symmetric if and only if $\g(S)=\frac{\F(S)+2}{2}$.
\end{enumerate}

\end{lem}

As a consequence of the above results we have that the set of symmetric (resp. pseudo-symmetric) numerical semigroups is the set of irreducible numerical semigroups with odd (resp. even) Frobenius number, and that a numerical semigroup, $S$, is irreducible if and only if $\g(S)=\left\lceil \frac{\F(S)+1}{2} \right\rceil$ (here $\lceil z \rceil  = \min\{n \in \Z: z\leq n\}$ the ceiling part of any rational number $z$).

\begin{prop}
\label{prop:5}
 Let $S$ be an irreducible numerical semigroup with Frobenius number $F$ and let $x$ be a minimal generator of $S$ verifying that:
\begin{enumerate}
 \item\label{prop5:c1} $x < F$,
\item\label{prop5:c2} $2x - F \not\in S$,
\item\label{prop5:c3} $3x \neq 2F$, and
\item\label{prop5:c4} $4x \neq 3F$.
\end{enumerate}
Then, $\overline{S}= \left( S \backslash \{x\}\right) \cup \{F - x\}$ is an irreducible numerical semigroup with Frobenius number $F$.
\end{prop}
\begin{proof}
 To prove the result it is enough to see that $\overline{S}$ is closed under addition since in that case $\overline{S}$ is a numerical semigroup with Frobenius number $F$ and $\g(\overline{S}) = \g(S)$. Hence, by Lemma \ref{lem:4} we get that $\overline{S}$ is irreducible.

The addition of two elements in $S\backslash \{x\}$ is an element in $S\backslash \{x\}$ (Lemma \ref{lem:1}). By Lemma \ref{lem:2} we deduce that $F-x \in \PF(S\backslash \{x\})$ and then $F-x + s \in \overline{S}$ for all $s \in S\backslash \{x\}$.

To conclude the proof, we see that $2(F-x) \in S\backslash \{x\}$. For the sake of that we distinguish two cases:
\begin{itemize}
 \item Assume that $S$ is symmetric. Since $2x-F \not\in S$, then by Lemma \ref{lem:3} we have that $F - (2x-F) \in S$. Thus, $2(F-x) \in S$. Moreover, since  $3x \neq 2F$, then $2(F-x) \neq x$ and consequently $2(F-x) \in S \backslash \{x\}$.
\item Assume now that $S$ is pseudo-symmetric. Since $2x-F \not\in S$, by Lemma \ref{lem:3} we have that either $2x-F = \frac{F}{2}$ or $F-(2x-F) \in S$. Since $4x \neq 3F$ then $2x-F \neq \frac{F}{2}$. Thus, $2(F-x) \in S$. Furthermore,  since  $3x \neq 2F$, then $2(F-x) \neq x$ and consequently $2(F-x) \in S \backslash \{x\}$.
\end{itemize}

\end{proof}

The smallest positive integer belonging to a numerical semigroup $S$ is called the  \emph{multiplicity} of $S$ and it is denoted by $\m(S)$. Note that $\m(S)$ is always a minimal generator of $S$.

\begin{cor}
 \label{cor:6}
Let $S$ be an irreducible numerical semigroup  with $\m(S) < \frac{\F(S)}{2}$. Then, $\overline{S}= \left( S \backslash \{\m(S)\}\right) \cup \{\F(S) - \m(S)\}$ is an irreducible numerical semigroup with $\F(\overline{S}) = \F(S)$ and $\m(\overline{S}) > \m(S)$.
\end{cor}

\begin{proof}
 Let us see that $\m(S)$ is a minimal generator of $S$ that holds the conditions \eqref{prop5:c1}-\eqref{prop5:c4} in the above proposition. Indeed, since $\m(S) < \frac{\F(S)}{2}$ then \eqref{prop5:c1} is clearly satisfied. Furthermore, $2\m(S) - \F(S) < 0$ and then \eqref{prop5:c2} is also verified. Also, since $\m(S) < \frac{\F(S)}{2}$, \eqref{prop5:c3} and \eqref{prop5:c4} are easily deduced. Finally, $\m(\overline{S}) > \m(S)$ because $\overline{S}= \left( S \backslash \{\m(S)\}\right) \cup \{\F(S) - \m(S)\}$ and $\F(S) - \m(S) > \m(S)$.
\end{proof}

Note that if we take an irreducible numerical semigroup and we apply repeatedly the construction given in Corollary \ref{cor:6}, we get, in a finite number of steps, an irreducible numerical semigroup with all its minimal generators larger than half of its Frobenius number. The following result shows that such a semigroup is unique.

\begin{prop}
\label{prop:7}
Let $F$ be a positive integer. Then, there exists an unique irreducible numerical semigroup $\C(F)$ with Frobenius number $F$ and all its minimal generators larger than $\frac{F}{2}$. Moreover,
$$
\C(F)=
 \left\{\begin{array}{rl}
\{0, \frac{F+1}{2}, \rightarrow\} \backslash \{F\} & \mbox{if $F$ is odd,}\\
\{0, \frac{F}{2}+1, \rightarrow\} \backslash \{F\} & \mbox{if $F$ is even,}
\end{array}\right.
$$
(here, if $a_1 < \cdots < a_k$ are integers, then we use $\{a_1, \ldots, a_k, \rightarrow\}$ to denote the set $\{a_1, \ldots, a_k\} \cup \{z \in \Z: z > a_k\}$.)
\end{prop}
\begin{proof}
It is clear that $\C(F)$ is a numerical semigroup with Frobenius number $F$. Furthermore, $\g(\C(F)) = \frac{F+1}{2}$ if $F$ is odd and $\g(\C(F)) = \frac{F+2}{2}$ if $F$ is even. By applying Lemma \ref{lem:4} we have that $\C(F)$ is irreducible.

Assume now that $S$ is an irreducible numerical semigroup with Frobenius number $F$ and with all its minimal generators larger than $\frac{F}{2}$, then, it is clear that $S \subseteq \C(F)$. By applying now that $S$ is maximal in the set of numerical semigroups with Frobenius number $F$, we conclude that $S=\C(F)$.
\end{proof}

For a given $S \in \I(F)$, we define the following sequence of elements in $\I(F)$:
\begin{itemize}
\item[] $S_0=S$,
\item[] $S_{n+1}= \left\{\begin{array}{cl}
\left( S_n \backslash \{\m(S_n)\}\right) \cup \{F - \m(S_n)\} & \mbox{if $\m(S_n) < \frac{F}{2}$,}\\
S_n & \mbox{ otherwise.}
\end{array}\right.
$
\end{itemize}
Let $k = \# \left\{s \in S: s < \frac{F}{2}\right\} -1$ (here $\#A$ stands for the cardinal of the finite set $A$), then $S_{k+n}=\C(F)$ for all $n \in \N$.

\begin{ex}
\label{ex:8}
Let $S=\langle 4, 5\rangle$, which is an irreducible numerical semigroup with Frobenius number $11$. Then:
\begin{itemize}
\item[] $S_0=\langle 4, 5\rangle$,
\item[] $S_1= \left( \langle 4, 5\rangle \backslash \{4\}\right) \backslash \{7\} = \langle 5,7,8,9\rangle$,
\item[] $S_2= \left( \langle  5,7,8,9 \rangle \backslash \{5\}\right) \backslash \{6\} = \langle 6,7,8,9,10\rangle = \C(11)$,
\item[] $S_{2+n} = \C(11) = \langle 6,7,8,9,10\rangle$, for all $n \in \N$.
\end{itemize}
Note that in this case $k = \# \left\{s \in S: s < \frac{11}{2}\right\} -1 = \#\{0,4,5\} - 1 = 2$.
\end{ex}

A (directed) graph $G$ is a pair $(V, E)$ where $V$ is a nonempty set whose elements are called vertices and $E$ is a subset of $\{(v,w) \in V\times V: v \neq w\}$. The elements in $E$ are called edges. A path that connects two vertices of $G$, $v$ and $w$, is a sequence of distinct edges in the form $(v_0, v_1), (v_1, v_2), \ldots, (v_{n-1}, v_n)$ with $v_0=v$ and $v_n=w$ for some $n \in \N$.

A graph $G$ is a tree if there exists a vertex $r$ (called the root of $G$) such that for any other vertex of $G$, $v$, there exists an unique path connecting $v$ and $r$. If $(v,w)$ is a edge of the tree, then, it is said that $v$ is a child of $w$. A leaf is a vertex without children.

Let $\I(F)$ be the set of irreducible numerical semigroups with Frobenius number $F$. We define the directed graph $\G(\I(F)) = (V, E)$ as follows:
\begin{itemize}
\item $V= \I(F)$, and
\item $(T,S) \in E$ if $m(T) < \frac{F}{2}$ and $S=\left( T \backslash \{\m(T)\}\right) \cup \{F - \m(T)\}$.
\end{itemize}

\begin{theo}
\label{theo:9}
Let $F$ be a positive integer. The directed graph $\G(\I(F))$ is a tree with root $\C(F)$. Moreover, if $S \in \I(F)$, then, the children of $S$ in $\G(\I(F))$ are $\left( S \backslash \{x_1\}\right) \cup \{F - x_1\}, \ldots, \left( S \backslash \{x_r\}\right) \cup \{F - x_r\}$, where $\{x_1, \ldots, x_r\}$ is the set of minimal generators, $x$, of $S$ that verify:
\begin{enumerate}
\item\label{theo9:c1} $\frac{F}{2} < x < F$,
\item\label{theo9:c2} $2x - F \not\in S$,
\item\label{theo9:c3} $3x \neq 2F$,
\item\label{theo9:c4} $4x \neq 3F$,
\item\label{theo9:c5} $F-x < \m(S)$.
\end{enumerate}
\end{theo}
\begin{proof}
As a direct consequence of the construction after Proposition \ref{prop:7}, we have that $\G(\I(F))$ is a tree with root $\C(F)$.

(\emph{Sufficiency}) By Proposition \ref{prop:5} we know that $T=\left( S \backslash \{x\}\right) \cup \{F - x\} \in \I(F)$. By condition \eqref{theo9:c5} we have that $\m(T)=F-x$ and by condition \eqref{theo9:c1} that $\m(T)<\frac{F}{2}$. Hence, $S= \left( T \backslash \{\m(T)\}\right) \cup \{F - \m(T)\}$, so $T$ is a child of $S$.

(\emph{Necessity}) Let $T$ be a child of $S$. Then, $\m(T) < \frac{F}{2}$ and $S= \left( T \backslash \{\m(T)\}\right) \cup \{F - \m(T)\}$. Then, $T=\left( S \backslash \{F - \m(T)\}\right) \cup \{F -(F - \m(T))\}$. To conclude the proof it is enough to see that $F-\m(T)$ is a minimal generator if $S$ verifying conditions \eqref{theo9:c1}-\eqref{theo9:c5}:
\begin{enumerate}
\item Since $F-\m(T) \not\in T$ and $S=\left(T \backslash \{\m(T)\}\right) \cup \{F - \m(T)\}$, we easy deduce that $F-\m(T)$ is a minimal generator of $S$. Furthermore, since $\m(T) < \frac{F}{2}$ then $\frac{F}{2} < F - \m(T) < F$.
\item If $2(F-\m(T)) - F \in S$, we have that $F-2\m(T) \in S$. However, $2\m(T) \in S$ since $2\m(T) \in T \backslash \{\m(T)\}$, we have that $F= F - 2\m(T) + 2\m(T) \in S$ which is not possible.
\item If $3(F-\m(T)) = 2F$ then $F= 3\m(T) \in S$ contradicting the definition of Frobenius number.
\item If $4(F-\m(T)) = 3F$ then $F= 4\m(T) \in S$ which is not possible.
\item Since $\m(T) < \frac{F}{2}$ then $\m(T) < F  - \m(T)$ and then, $\m(T) < \m(S)$. Thus, $F - (F-\m(T)) < \m(S)$.
\end{enumerate}
\end{proof}

We conclude this section by illustrating the applicability of the above result to construct, explicitly, the tree $\G(\I(F))$.

\begin{ex}
\label{ex:10}
Let us compute the whole set of irreducible numerical semigroups with Frobenius number $11$. We start by the root of the tree, $\C(11)=\langle 6,7,8,9,10 \rangle$. The minimal generators that hold the conditions \eqref{theo9:c1}-\eqref{theo9:c5} are: $x_1=8$, $x_2=7$ and $x_3=6$. Then, the children of $\C(11)$ are:
\begin{itemize}
\item[] $\left(\langle 6,7,8,9,10 \rangle \backslash \{8\} \right) \cup \{3\} = \langle 3, 7 \rangle$
\item[] $\left(\langle 6,7,8,9,10 \rangle \backslash \{7\} \right) \cup \{4\} = \langle 4,6,9 \rangle$
\item[] $\left(\langle 6,7,8,9,10 \rangle \backslash \{6\} \right) \cup \{5\} = \langle 5,7,8,9 \rangle$
\end{itemize}
Thus, the tree $\G(\I(11))$ starts in the form:
$$
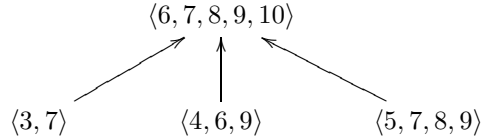
\xymatrix{ & \langle 6,7,8,9,10 \rangle & \\
\langle 3, 7 \rangle \ar[ru] & \langle 4,6,9 \rangle \ar[u]& \langle 5,7,8,9 \rangle \ar[lu]}
$$

Next, we compute the children of $\langle 3, 7 \rangle$, $\langle 4,6,9 \rangle$ and $\langle 5,7,8,9 \rangle$. First, we observe that $\langle 3, 7 \rangle$ has no children. The unique generator of $\langle 4,6,9 \rangle$ with the conditions of Theorem \ref{theo:9} is $x_1=9$. Then, $\langle 4,6,9 \rangle$ only has a child, which is $\left( \langle 4,6,9 \rangle \backslash \{9\}\right) \cup \{2\} = \langle 2, 13\rangle$. For $\langle 5,7,8,9 \rangle$, only $x_1=7$ holds the conditions of Theorem \ref{theo:9}, so the unique child of $\langle 5,7,8,9 \rangle$ is $\left( \langle 5,7,8,9 \rangle \backslash \{7\}\right) \cup \{4\} = \langle 4, 5\rangle$. Hence, the tree continues as:
$$
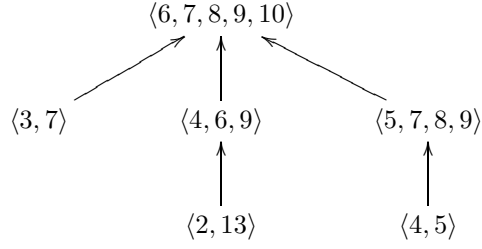
\xymatrix{ & \langle 6,7,8,9,10 \rangle & \\
\langle 3, 7 \rangle \ar[ru] & \langle 4,6,9 \rangle \ar[u]& \langle 5,7,8,9 \rangle \ar[lu]\\
& \langle 2, 13\rangle \ar[u]& \langle 4,5 \rangle \ar[u]}
$$

Finally, since $\langle 2,3\rangle$ and $\langle 4, 5 \rangle$ do not have  minimal generators with the conditions of Theorem \ref{theo:9}, they have no children. Thus, the tree $\G(\I(11))$ is completed.
\end{ex}

\section{A Kunz-coordinates vector construction of the tree $\G(\I(F))$}
\label{sec:2}

In this section we use a different encoding of a numerical semigroups to compute the tree $\G(\I(F))$ for any positive integer $F$. We present an analogous construction to the one in the section above but based on manipulating vectors in $\{0,1\}^F$. It leads us to an efficient procedure to compute the set of irreducible numerical semigroups with Frobenius number $F$.

 Let $S$ be a numerical semigroup and $n \in S$. The \emph{Ap\'ery set} of $S$ with respect to $n$ is the set $\Ap(S,n) = \{s
\in S :  s - n \not\in S\}$. This set was introduced by Ap\'ery in \cite{apery}.

The following characterization  of the Ap\'ery set that appears in \cite{springer} will be useful for our development.
 \begin{lem}
 \label{lem:11a}
Let $S$ be a numerical semigroup and $n \in S\backslash \{0\}$. Then $\Ap(S,n) = \{0 = w_0, w_1, \ldots,  w_{n -
1}\}$, where $w_i$ is the least
element in $S$ congruent with $i$ modulo $n$, for $i=1, \ldots, n-1$.
 \end{lem}

Moreover, the set $\Ap(S,n)$ completely determines $S$, since $S =
\langle \Ap(S,n) \cup \{n\} \rangle$ (see \cite{london}), and then, we can identify $S$ with its Ap\'ery set with respect to $n$. The set $\Ap(S,n)$ contains, in general, more information than an arbitrary system of
generators of $S$. For instance, Selmer in \cite{selmer77} gives the formulas,
$\g(S)=\frac{1}{n}\left(\sum_{w \in \Ap(S, n)} w\right) -
\frac{n-1}{2}$ and $\F(S) = \max(\Ap(S,n)) - n$. One can also test if a nonnegative integer $s$ belongs to $S$ by checking if $w_{s\pmod m} \leq s$.

We consider an useful modification of the Ap\'ery set that we call the \emph{Kunz-coordinates vector} as in \cite{bp2011}. Let $S$ be a numerical semigroup and $n \in S\backslash\{0\}$. By Lemma \ref{lem:11a}, $\Ap(S, n)=\{w_0=0, w_1, \ldots, w_{n-1}\}$, with $w_i$ congruent with $i$ modulo $n$. The \emph{Kunz-coordinates vector} of $S$ with respect to $n$ is the vector $\Kunz(S,n) = x \in \N^{n-1}$ with components $x_i = \frac{w_i-i}{n}$ for $i=1, \ldots, n-1$. We say that $x\in \N^{n-1}$ is a Kunz-coordinates vector with respect to $n$ if there exists a numerical semigroup whose Kunz-coordinates vector is $x$. If $x \in \N^{n-1}$ is a Kunz-coordinates vector, we denote by $S_x$ the numerical semigroup such that $\Kunz(S_x, n) = x$.

The Kunz-coordinates vectors were introduced in \cite{kunz} and have been previously analyzed when $n=\m(S)$ in \cite{bp2011,london}.

In the following result, whose proof is immediate from the results in the section above, it is shown how to get the multiplicity and the genus of a numerical semigroup with Frobenius number $F$ from its Kunz-coordinates vector with respect to $F+1$.
\begin{lem}
\label{lem:11}
Let $S$ be a numerical semigroup and $x=\Kunz(S, \F(S)+1) \in \N^F$. Then:
\begin{enumerate}
\item $ x \in \{0,1\}^{\F(S)}$,
\item $\G(S) = \{i \in \{1, \ldots, \F(S)\}: x_i =1\}$,
\item $\m(S)= \left\{\begin{array}{cl}
\min\{i: x_i=0\} & \mbox{if $\{i: x_i=0\} \neq \emptyset$},\\
\F(S)+1 & \mbox{otherwise}\end{array}\right.$,
\item $\g(S)= \dsum_{i=1}^F x_i$.
\end{enumerate}
\end{lem}

We denote by $\SG(S)=\{h \in \N \backslash S: S \cup \{h\} \text{ is a numerical semigroup}\}$ the set of special gaps of the numerical semigroup $S$. It is well-known (see \cite{springer}) that $S$ is irreducible if and only if $\#\SG(S)=1$. Furthermore, $\SG(S)$ can be computed from $x=\Kunz(S, \F(S)+1)$ by applying the characterization of $\SG(S)$ in terms of the Ap\'ery set that appears in \cite{jpaa}:
\begin{align*}
\SG(S) &= \big\{i \in \{1, \ldots, \F(S)\}:& &x_i =1,\\
&&& x_j \geq x_{i+j} \quad \text{, for all } j=i, \ldots, \F(S)-i \text{, and}\\
&&&x_{2i}=0 \text{ if $i < \frac{\F(S)}{2}$}\big\} 
\end{align*}

The construction of the tree of irreducible numerical semigroups with fixed Frobenius number in the section above consists of adding and removing certain elements to the semigroups to get the child of an irreducible numerical semigroup (Theorem \ref{theo:9}). The following lemma, whose proof is trivial, informs us about the translations of adding and removing elements of a numerical semigroups in its Kunz-coordinates vector. 
\begin{lem}
\label{lem:12}
Let $S$ be a numerical semigroup and $\e_i$ the $\F(S)$-tuple having a $1$ as its $i$th entry and zeros otherwise, for $i=1, \ldots, \F(S)$. Then:
\begin{enumerate}
\item If $h \in \SG(S)\backslash \{\F(S)\}$, then $\Kunz(S \cup \{h\}, \F(S)+1) = \Kunz(S, \F(S)+1) - \e_{h}$.
\item If $n <\F(S)+1$ is a minimal generator of $S$, then $\Kunz(S \backslash \{n\}, \F(S)+1) = \Kunz(S, \F(S)+1) + \e_{n}$.
\end{enumerate}
\end{lem}

In the section above we proved that there is a numerical semigroup of special importance for the construction of the tree $\G(\I(F))$, $\C(F)$. The following result shows the Kunz-coordinates vector of this semigroup with respect to $F+1$.
\begin{lem}
\label{lem:14}
Let $F$ be a positive integer. Then, $\Kunz(\C(F), F+1) = (\overbrace{1, \ldots, 1}^{\lceil \frac{F+1}{2}\rceil -1}, \overbrace{0, \ldots, 0}^{F-\lceil\frac{F+1}{2}\rceil}, 1) \in \N^F$.
\end{lem}

For any positive integer $F$, we define the directed graph $\G^K(\I(F))$ as the translation of $\G(\I(F))$ in terms of the Kunz-coordinates vectors with respect to $F+1$. The vertices of $\G^K(\I(F))$ is the set $\{\Kunz(S, F+1): S \in \I(F)\}$, and $(x,y)$ is an edge of $\G^K(\I(F))$ if $m =\min\{i: x_i=0\} < \frac{F}{2}$ and $y = x + \e_m - \e_{F-m}$ (Lemma \ref{lem:12}).

Now, we proceed to translate the conditions of Theorem \ref{theo:9} in terms of the Kunz-coordinates vector.
\begin{theo}
\label{theo:15}
Let $F$ be a positive integer. The directed graph $\G^K(\I(F))$ is a tree with root $\hat{x}=(\overbrace{1, \ldots, 1}^{\lceil \frac{F+1}{2}\rceil -1}, \overbrace{0, \ldots, 0}^{F-\lceil\frac{F+1}{2}\rceil}, 1)$. Moreover, if $x$ is the Kunz-coordinates vector with respect to $F$ of an irreducible numerical semigroup with Frobenius number $F$, then, the children of $x$ in $\G^K(\I(F))$ are $x+\e_{n_1}-\e_{F-n_1}, \ldots, x+\e_{n_r}-\e_{F-n_r}$, where $\{n_1, \ldots, n_r\}$  verifies:
\begin{enumerate}
\item $x_{n_i}=0$, for all $i=1, \ldots, r$,
\item $x_k + x_j \geq 1$ if $k+j=n_i$, for all $i=1, \ldots, r$,
\item $\frac{F}{2} < n_i < F$, for all $i=1, \ldots, r$,
\item $x_{2n_i-F} = 1$, for all $i=1, \ldots, r$,
\item $3n_i \neq 2F$, for all $i=1, \ldots, r$,
\item $4n_i \neq 3F$, for all $i=1, \ldots, r$,
\item $n_i > F - \min\{j: x_j=0\}$.
\end{enumerate}
\end{theo}

\begin{proof}
Note that conditions $(1)$ and $(2)$ are equivalent to say that $n_i$ is a minimal generator smaller than $F$ of $S_x$. %$(1)$ follows from Lemma \ref{lem:13} and $(2)$ consists of checking whether $x+\e_{n_i}$ verifies the conditions to be a Kunz-coordinates vector (and then, encoding a numerical semigroup).

Items $(3)-(7)$ are the translations of conditions $\eqref{theo9:c1}-\eqref{theo9:c5}$ in Theorem \ref{theo:9}. 
\end{proof}
Note that the satisfaction of conditions $(1)-(7)$ in Theorem \ref{theo:15} is checked by analyzing only the connections between the $F$ components of a $0-1$ vector. From a computational viewpoint this is much easier than checking conditions \eqref{theo9:c1}-\eqref{theo9:c5} in Theorem \ref{theo:9}. Observe also that $x + \e_{n} - e_{F-n}$ in our case ($x \in \{0,1\}^F$, $n$ a minimal generator and $F-n$ a special gap of $S_x$) consists of swapping $x_n$ with $x_{F-n}$. Hence, the computational effort is small at each step.

Other advantage of this approach is that in general, the computation of the Ap\'ery sets and the Kunz-coordinates vector of a numerical semigroup is hard. However, by the above theorem we do not need to compute the Kunz-coordinates vectors of the semigroups appearing in the tree $\G^K(\I(F))$ since by starting with $(\overbrace{1, \ldots, 1}^{\lceil \frac{F+1}{2}\rceil -1}, \overbrace{0, \ldots, 0}^{F-\lceil\frac{F+1}{2}\rceil}, 1)$ and manipulating adequately the $0-1$ vectors representing these semigroups we are assured to get Kunz-coordinates vectors of irreducible numerical semigroups with Frobenius number $F$. 

A vertex in $\G^K(\I))$ has at most $\frac{F}{2}$ children (the number of possible swappings when the parent has $\frac{F}{2}$ ones and $\frac{F}{2}$ zeroes). Then, the complexity of each step (building the children of a vertex) has computational complexity $O(F^2\; \delta(F))$ where $\delta(F)$ is the computational complexity of checking the conditions in Theorem \ref{theo:15}, which is clearly a polynomial in $F$. Also, the height of the tree is bounded above by $\frac{F}{2}$. Hence, computing the tree $\G^K(\I(F))$ is polynomial-time doable.
 
Note that the tree structure of the set of irreducible numerical semigroups with a fixed Frobenius number is useful when one is interested in finding irreducible numerical semigroups with some extra properties. For instance, this representation is useful when searching for the set of irreducible oversemigroups of a given numerical semigroup with a fixed Frobenius number or those that also has a fixed multiplicity (see \cite{bp2011}). 

The following example illustrates the construction of the tree $\G^K(\I(F))$ in the same case than in Example \ref{ex:10}.
\begin{ex}
\label{ex:16}
Let us compute the whole set of irreducible numerical semigroups with Frobenius number $F=11$ by using Kunz-coordinates. In this case $\lceil \frac{F+1}{2} \rceil = 6$. First we star with $x=(1,1,1,1,1, 0,0,0,0,0,1)$ which is the root of the tree $\G^K(\I(11))$. The set of indices verifying conditions of Theorem \ref{theo:15} is $\{6,7,8\}$. Note that to get this set, first we look for the components in $x$ that are equal to zero (Condition $(1)$) ,which are $\{6,7,8,9,10\}$; then we check for those elements that verify conditions $(3)$ and $(7)$ (no elements are removed in this step). From those we filter by those that verify Condition $(4)$ ($9$ and $10$ are removed from the set since ${x}_{7}={x}_9=0$). Finally we check for the rest of conditions, which are in this case verified by all the remainder elements. Therefore, the children of ${x}$ are $x^1={x}+\e_8-\e_3, x^2= {x}+\e_7-\e_4$ and $x^3={x}+\e_6-\e_5$. Then, the first part of the tree is:
{\small
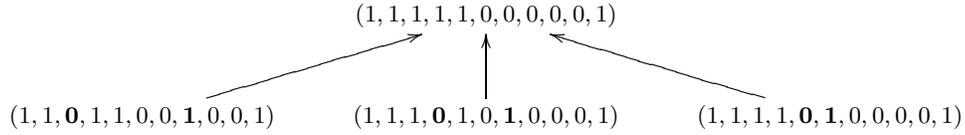
$$
\xymatrix{ & (1,1,1,1,1, 0,0,0,0,0,1) & \\
 (1,1,\mathbf{0},1,1, 0,0,\mathbf{1},0,0,1)\ar[ru] & (1,1,1,\mathbf{0},1, 0,\mathbf{1},0,0,0,1) \ar[u]& (1,1,1,1,\mathbf{0}, \mathbf{1},0,0,0,0,1) \ar[lu]}
$$}
The elements which are swapped to get each child are marked in boldface.

Now, we  compute the children of $x^1, x^2$ and $x^3$. By an analogous procedure we get that $x^1$ has no children, for $x^2$ the unique index verifying the conditions of Theorem \ref{theo:15} is $n_1=9$, being the unique child $x^4=x^2 +\e_9 - \e_2$. Also, $x^3$ has only a child $x^5=x^3 + \e_7-\e_4$ which comes from the unique index that verifies the conditions, $n_1=7$. Then, the continuation of the tree is:
{\small
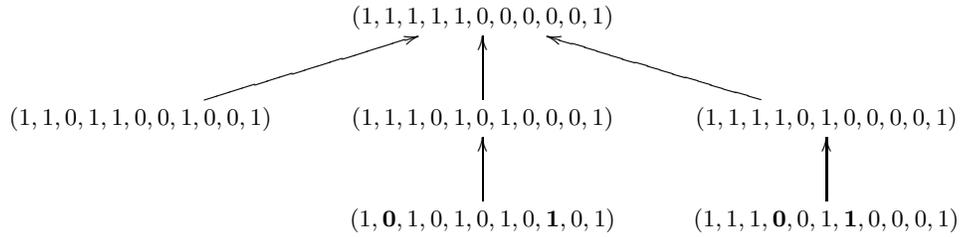
$$
\xymatrix{ & (1,1,1,1,1, 0,0,0,0,0,1) & \\
 (1,1,0,1,1, 0,0,1,0,0,1)\ar[ru] & (1,1,1,0,1, 0,1,0,0,0,1) \ar[u]& (1,1,1,1,0, 1,0,0,0,0,1) \ar[lu]\\
 & \ar[u] (1,\mathbf{0},1,0,1, 0,1,0,\mathbf{1},0,1) & (1,1,1,\mathbf{0},0, 1,\mathbf{1},0,0,0,1) \ar[u] }
$$}
Finally, we see that $x^4$ and $x^5$ have no children, so the tree $\G^K(\I(11))$ is completed.
\end{ex}

\end{document}